\documentclass[sn-mathphys-num, 12pt]{sn-jnl}

\usepackage{amsmath,amssymb,amsfonts,amsthm,mathrsfs, latexsym, textcomp,amscd}%
\usepackage{multirow}
\usepackage[title]{appendix}
\usepackage{textcomp}
\usepackage{manyfoot}
\usepackage{booktabs}
\usepackage{algorithm}
\usepackage{algorithmicx}
\usepackage{algpseudocode}
\usepackage{listings}
\usepackage{natbib}
\usepackage[dvips]{epsfig}
\usepackage{graphics,verbatim}
\usepackage{upref, hyperref}
\usepackage{xcolor}
\usepackage{wrapfig}
\usepackage{enumitem} 
\usepackage[nameinlink]{cleveref}
\usepackage{orcidlink}
\usepackage{bbm}

\theoremstyle{thmstyleone}
\newtheorem{theorem}{Theorem}

\newtheorem{lemma}[theorem]{Lemma}

\theoremstyle{thmstyletwo}

\newtheorem{remark}{Remark}

\theoremstyle{thmstylethree}

\newcommand{\<}{\left\langle}
\renewcommand{\>}{\right\rangle}
\renewcommand{\(}{\left(}
\renewcommand{\)}{\right)}
\renewcommand{\[}{\left[}
\renewcommand{\]}{\right]}
\renewcommand{\{}{\left\lbrace }
\renewcommand{\}}{\right\rbrace }

\newcommand{\Be} {\begin{equation}}
	\newcommand{\Ee} {\end{equation}}
\newcommand{\Nee} {\notag\end{equation}}
\newcommand{\bea} {\begin{eqnarray}}
\newcommand{\eea} {\end{eqnarray}}
\newcommand{\Bea} {\begin{eqnarray*}}
\newcommand{\Eea} {\end{eqnarray*}}

\newcommand{\abs}[1]{\left\vert#1\right\vert}

\newcommand{\norm}[1]{\left\Vert#1\right\Vert}

\newcommand{\De} {\Delta}
\newcommand{\la} {\lambda}
\newcommand{\rn}{\mathbb{R}^{N}}
\newcommand{\hn}{\mathbb{H}^{N}}

\newcommand{\N}{\ensuremath{\mathbb{N}}}
\newcommand{\R}{\ensuremath{\mathbb{R}}}
\newcommand{\rchi}{\protect\raisebox{2pt}{$\chi$}}

\newcommand{\eps}{\varepsilon}
\newcommand{\dv}{\mathrm{~d} V_{\bn}}
\newcommand{\ex}{\mathrm{exp_0}}

\newcommand{\al} {\alpha}

\newcommand{\de} {\delta}

\newcommand{\Ga} {\Gamma}

\newcommand{\f}{\frac}

\newcommand{\ef}{\eqref}

\newcommand{\bn}{\mathbb{B}^{N}}

\begin{document}


\title[Symmetry breaking solutions]{Existence of a bi-radial sign-changing solution for Hardy-Sobolev-Mazya type equation}

\author[1]{\fnm{Atanu} \sur{Manna} \orcidlink{0009-0001-3419-3440}}

\author[2]{\fnm{Bhakti Bhusan} \sur{Manna} \orcidlink{0000-0002-8098-2782}}


%


\abstract{ 	
	In this article, we study the following Hardy-Sobolev-Maz'ya type equation:
	\begin{align*}
			-\Delta u - \mu \frac{u}{|z|^2} = \frac{\abs{u}^{q - 2}u}{|z|^t}, \quad u \in D^{1,2} \(\mathbb{R}^n\), 
	\end{align*}
	where $x = (y,z) \in \mathbb{R}^h \times \mathbb{R}^k = \mathbb{R}^n$, with $n \geq 5$, $2 < k <n$, and $t = n - \frac{(n-2)q}{2}$. We establish the existence of a bi-radial sign-changing solution under the assumptions $0 \leq \mu < \frac{(k-2)^2}{4}, \, 2 < q <2^* = \frac{2(n-k+1)}{n-k-1}$. We approach the problem by lifting it to the hyperbolic setting, leading to the equation: $-\Delta_{\mathbb{B}^N} u  \, - \,  \lambda  u = |u|^{p-1}u, \; u \in H^1\(\mathbb{B}^N\)$, $\mathbb{B}^N$ is the hyperbolic ball model. We study the existence of a sign-changing solution with suitable symmetry by constructing an appropriate invariant subspace of $H^1\(\mathbb{B}^N\)$ and applying the concentration compactness principle, and the corresponding solution of the Hardy-Sobolev-Maz'ya type equation becomes bi-radial under the corresponding isometry.
}





\footnotetext[1]{\textbf{Atanu Manna}, IIT Hyderabad, Department of Mathematics, Kandi, Sangareddy, Telangana, 502284, India, email: mannaatanu98@gmail.com}

\footnotetext[2]{\textbf{ Bhakti Bhusan Manna}, IIT Hyderabad, Department of Mathematics, Kandi, Sangareddy, Telangana, 502284, India, email: bbmanna@math.iith.ac.in}

\maketitle


\section{Introduction}

This article focuses on establishing the existence of bi-radial sign-changing solutions to the Euler-Lagrange equation for the optimal Hardy-Sobolev-Maz'ya inequality \cite{MR817985}, in the subcritical case. 

Let $(y,z) \in \R^h \times \R^k = \R^n$, with $n \geq 3$ and $\mu \leq \f{(k-2)^2}{4}$. Then for all $u \in C_0^\infty \(\R^h \times \(\R^k \setminus \{0\}\)\)$, the Hardy-Sobolev-Maz'ya inequality is
\begin{align}
	S_t^\mu \( \int_{\R^h \times \R^k}  \frac{\abs{u}^q}{\abs{z}^t} \mathrm{~d}y \mathrm{~d}z \)^{\f{2}{q}} \leq \int_{\R^h \times \R^k} \[\abs{\nabla u}^2 - \mu \f{u^2}{\abs{z}^2}\]  \mathrm{~d}y \mathrm{~d}z, \label{HSM inequality}
\end{align}
where $2 < q \leq \frac{2n}{n-2}, \; t = n - \f{(n-2)q}{2}$. The constant $S_t^\mu > 0$ is optimal, and we take $\R^+$ in place of $\R$ when $k=1$.

The corresponding Euler-Lagrange equation for (\ref{HSM inequality}) is
\begin{align}
	- \De u - \mu \f{u}{|z|^2} = \f{\abs{u}^{q - 2}u}{|z|^t}, \quad u \in D^{1,2} \(\R^n\) \quad \text{in} \; \R^n \;\; \( \; \text{in} \; \R^h \times \R^+ \; \text{if} \; k =1\) \label{HSM equation} \tag{$H$}
\end{align}
Existence of minimizers for (\ref{HSM inequality}) has been studied in \cite{MR1918928, MR2416099, MR2317783} for various cases. In case $k \geq 2$ and $\mu = 0$, cylindrical symmetry of positive extremals of (\ref{HSM equation}) has been studied in \cite{MR2223717}. For $k \geq 2, \; 2 < q <\f{2(n-k+1)}{n-k-1}, \; 0 \leq \mu \leq \f{(k-2)^2}{4}$ the authors in \cite{MR2589572} have shown that there exists at least an entire cylindrically symmetric positive solution to (\ref{HSM equation}). Also, it was established that for $q < \f{2n}{n-2}$ and $ \mu <  \f{(k-2)^2}{4} - \f{k-1}{q-2}$, ground state positive solutions are not cylindrically symmetric. For the definition of cylindrically symmetric function, see section 1 in \cite{MR2589572}.

In \cite{MR2474591, MR2483639}, the authors have found a connection between the solutions to (\ref{HSM equation}), which are symmetric in the $z$-variable, and the solutions to certain elliptic equation on some hyperbolic space. In particular, let $N = h +1 = n- k+1, p = q - 1$, and $\la = \mu +  \f{h^2 - (k-2)^2}{4}$, then $u(y,z) = u(y, \abs{z})$ solves (\ref{HSM equation}) if and only if $v = w \circ M$ solves
\begin{align*}
	-\Delta_{\mathbb{B}^N} u  \, - \,  \lambda  u &= |u|^{p-1}u, \quad \; u \in H^1\(\bn\), \qquad \label{(1)} \tag{$Eq_\la$}
\end{align*}
where $w$ is given by $w(y,r) = r^{\f{n-2}{2}} u (y, r)$ and $\De_{\bn}$ is the Laplace-Beltrami operator on hyperbolic space $\bn$. Also, $M$ is an isometry from the Poincar\'e ball model of hyperbolic space $\bn$ onto the upper half space model $\hn$, given as in (\ref{Isometry between two models}).

In \cite{MR2483639}, the authors have studied (\ref{(1)}) for various existence and non-existence results for positive solutions. In particular, they have established that for $N \geq 3, 1 < p < 2^* - 1 = \f{N+2}{N-2}, \la < \f{(N-1)^2}{4}$, there exists a unique (upto hyperbolic isometries) positive solution. Here we note that $0 \leq \mu \leq \f{(k-2)^2}{4}$ implies $ \f{(N-1)^2}{4} - \f{(k-2)^2}{4} \leq \la \leq \f{(N-1)^2}{4}$. Furthermore, using the moving plane method, it was established that every positive solution of (\ref{(1)}) has hyperbolic symmetry, i.e., it is radial in hyperbolic space. This yields a corresponding existence result for cylindrical symmetric solutions to (\ref{HSM equation}) using the above-discussed relationship.

Later, the authors in \cite{MR2898778}, proved that when $p < 2^* - 1$, there exist infinitely many radial sign-changing solutions to (\ref{(1)}), employing a Strauss type argument to establish the compact embedding of $H_r^1\(\bn\)$ into $L^{p+1}\(\bn\)$. Here $H_r^1\(\bn\)$ refers to the subspace of $H^1\(\bn\)$ that comprises solely radial functions. Again, using the relationship between (\ref{HSM equation}) and (\ref{(1)}), it was established that the equation (\ref{HSM equation}) has infinitely many sign-changing solutions when $q < 2^*$.

In contrast to earlier works that primarily addressed cylindrically symmetric or hyperbolically radial solutions to (\ref{HSM equation}), we focus on constructing sign-changing solutions that reflect a richer symmetry, namely the bi-radiality. 
\begin{theorem} \label{Existence thm HSM}
	Let $n \geq 5, 2 < k < n$ and $0 \leq \mu < \f{(k-2)^2}{4}, 2 < q <2^*$; then equation (\ref{HSM equation}) admits a bi-radial sign-changing solution.
\end{theorem}\;

This establishes the existence of bi-radial sign-changing solutions to (\ref{HSM equation}), expanding the knowledge of known solutions and revealing a new geometric structure of solutions to the Hardy-Sobolev-Maz'ya type equation.


To prove this, we use the fact that the hyperbolic counterpart of (\ref{HSM equation}) has a sign-changing solution with certain symmetry properties corresponding to the following isometric actions on $\bn$:
\begin{align}\label{isom.grp}
G & :=\{ \begin{bmatrix}
		A & 0\\
		0 & -1
	\end{bmatrix} \; : \; A \in O(N-1)\}
\end{align}
In particular, we prove the following theorem:

\begin{theorem} \label{Existence thm hyperbolic}
	For $N \geq 3, \la < \f{(N-1)^2}{4}, 1 < p < 2^* -1$, (\ref{(1)}) admits a non-radial sign-changing solution $u$ such that $u(gx) = - u(x), \, \forall g \in G$.
\end{theorem}\;

In our recent work \cite{MannaManna+2025}, we demonstrated the existence of non-radial sign-changing solutions for \eqref{(1)} when $N=5$, and infinitely many such solutions when $N = 4 \; \& \; N \geq 6$. This was achieved by constructing a closed subspace of $H^{1}(\bn)$ defined by symmetry constraints:
\begin{align}
	H^1(\bn)^\phi & =  \{ u\in H^1(\bn) \,:\, u \(gx\) = \phi\(g\) u\(x\), \; \forall g \in \Ga, x \in \bn \}, \label{phi equivariant subsapce}
\end{align}
where $\Ga$ is a compact Lie subgroup of $O(N) \subset \mathrm{Iso}(\bn)$ and $\phi : \Ga \to \mathbb{Z}_2$ is a continuous surjective homomorphism. Using the principle of symmetric criticality \cite{MR2083309}, critical points of the associated energy functional restricted to $H^1(\bn)^\phi$ yield solutions to (\ref{(1)}). For $ N = 4 \; \& \; N \geq 6$, the fountain theorem demonstrates the existence of infinitely many critical points for the associated energy functional. In the case of $N=5$, a $(PS)_c$ sequence at the mountain pass min-max level has been investigated, and a concentration compactness type argument is used to establish the existence of a non-trivial critical point of the energy functional. The key assumptions for this framework were:
\begin{equation}
	\exists \,\,\,x \in \bn \text{ \,  such \ that \,\, } \Ga_x \subset \mathrm{ker} \,\phi. \label{(A_1)} \tag{$\mathbf{A_1}$}
\end{equation}
And,
\begin{equation}
	\text{for every }  x\in \bn  \,, \text{ either    }  \#\Ga (x)=\infty \text{   or } \Ga (x)=\{x\}. \label{(A_2)} \tag{$\mathbf{A_2}$}
\end{equation}

Here $\Ga_x$ is the isotropy subgroup of $x$ and $\Ga(x)$ is the orbit of $x$. As a consequence of existence and multiplicity results in \cite{MannaManna+2025}, we obtain the following existence result for non-radial sign-changing solution to (\ref{HSM equation}):

\begin{theorem} \label{Multiplicity thm HSM}
	\begin{enumerate}
		\item[(a)] Let $n = 7, 2 < k < n$, and $0 \leq \mu < \f{(k-2)^2}{4}, 2 < q <2^*$; then equation (\ref{HSM equation}) admits a non-radial sign-changing solution.
		
		\item[(b)] Let $n \geq 8, 2 < k < n$, and $0 \leq \mu < \f{(k-2)^2}{4}, 2 < q <2^*$; then equation (\ref{HSM equation}) admits infinitely many non-radial sign-changing solutions.
	\end{enumerate}
\end{theorem}\;

To establish the existence of non-radial solutions to (\ref{(1)}) is particularly challenging for the lower dimension case $N=3$, as the techniques used for higher dimensions, relying on closed subgroups of $O(N)\subset \mathrm{Iso}(\bn)$, the isometric group of $\bn$- are not directly applicable. This limitation is observed in \cite{MR4068471} for the Euclidean setting, which prevents the straightforward application of symmetry-based variational methods used in our earlier study. It was observed that no closed subgroup $\Ga$ of $O(3)$ satisfies both (\ref{(A_1)}) and (\ref{(A_2)}) while admitting a surjective homomorphism $ \phi : \Ga \to \mathbb{Z}_2$. This obstacle prevents the use of the same variational techniques for $N=3$. 

Inspired by \cite{MR1896096}, which deals with non-radial sign-changing solutions of Schr\"odinger-Poisson systems in $\R^3$, we adopt a new symmetry approach to overcome this challenge in the hyperbolic space setting and prove the \cref{Existence thm hyperbolic}. Furthermore, we use the same symmetry to showcase the existence of a bi-radial sign-changing solution to (\ref{HSM equation}), as in \cref{Existence thm HSM}. Also, we observe that the solutions to (\ref{(1)}) established in  \cite{MannaManna+2025} and the solutions in \cref{Existence thm HSM} exhibit different kinds of symmetry, which in turn give the following multiplicity theorem for (\ref{(1)}):\\

\begin{theorem}\label{Multiplicity thm hyperbolic}
	For $N \geq 4, \la < \f{(N-1)^2}{4}, 1 < p < 2^* -1$, (\ref{(1)}) admits at least two non-radial sign-changing solutions.
\end{theorem}


\subsection{Notations and Preliminaries}


%
%

The Poincar\'e ball model of hyperbolic space, denoted by $\bn$, is defined as $B(0,1)\subset\rn$ equipped with the Riemannian metric:
$\(g_{\bn}\)_{ij} = \(\frac{2}{1 - |x|^2}\)^{2} \delta_{i j}.$ Another commonly used model for hyperbolic space is the upper half space model, denoted $\hn$, which consists of the upper half space $\rn_+ = \{\(x_1, \cdots, x_N\) \in \rn : x_N > 0\}$, endowed with the Riemannian metric $\(g_{\hn}\)_{ij} = \(\f{1}{x_N}\)^2 \de_{ij}.$

An isometry $M : \bn \to \hn$ between these two models is given by 
\begin{align}
	M \(x\) = M \((x', x_N)\)= \(\f{2x'}{\abs{x'}^2 + \(1 + x_N\)^2}, \f{1-\abs{x}^2}{\abs{x'}^2 + \(1 + x_N\)^2} \), \label{Isometry between two models}
\end{align}
with the property $M = M^{-1}$. We denote the hyperbolic distance between two points  $x,y \in \bn$ as $d_{\bn}(x,y)$ and has the form
\begin{align}
	d_{\bn}(x,y) & := \cosh^{-1} \( 1 + \f{2|x-y|^2}{\(1-|x|^2\)\(1-|y|^2\)}\). \label{hyperbolic metric}
\end{align}
We denote by $B \(x,r\) := \{ y \in \bn : d_{\bn}(x,y) < r\}$, the hyperbolic ball centered at $x \in \bn$ with radius $r > 0$, and the corresponding hyperbolic sphere by $S \(x,r\) := \{ y \in \bn : d_{\bn}(x,y) = r\}.$ Here we remark that, if the center is at the origin, the hyperbolic and Euclidean spheres coincide, i.e., $S \(0,r\) = S_E \(0, \tanh \(\f{r}{2}\)\),$  where $S_E \(0, \tanh \(\f{r}{2}\)\)$ denotes a Euclidean sphere of radius $\tanh \(\f{r}{2}\)$ centered at the origin.
\vspace*{0.2cm}
\newline
\textbf{Exponential map on $\bn$:} Given that hyperbolic space is a complete manifold, it follows from the Hopf-Rinow theorem that $\bn$ is geodesically complete. Consequently, the exponential map at the origin $\mathrm{exp}_0: T_0(\bn) \to \bn$ is well defined and expressed as
\begin{align}
	\mathrm{exp}_0(z) & = \f{\sinh \(2|z|\)}{1+ \cosh \(2|z|\)} \f{z}{|z|},  \quad \forall z \in T_0 \(\bn\), \label{exponential map}
\end{align}
where $|z|$ is the Euclidean norm in $\rn$. Since the injectivity radius of $\bn$ at the origin is infinite, $\mathrm{exp}_0$ is a diffeomorphism from the tangent space $T_0\(\bn\)$ onto $\bn$. Furthermore, for any $r>0$, the map $\mathrm{exp}_0 : B_{T_0\(\bn\)} \(z,r\)  \to B\(\ex(z),r\)$ is also a diffeomorphism, where $B_{T_0\(\bn\)} \(z,r\)$ refers to a ball with center at $0$ and radius $r$ inside the tangent space $T_0\(\bn\)$. The tangent space $T_0\(\bn\)$ can be identified as the Euclidean space $\rn$, and the metric in tangent space is the same as Euclidean distance. For convenience, we denote $B_{T_0\(\bn\)} \(z,r\)$ as $B_E(z,r)$. We also use the following notations:
\begin{align*}
	\overline{B\(\ex(z),r\)} & = \ex\(\overline{B_E(z,r)}\), \quad \text{for}\; r>0.\\
	A_E \(z; r_2,r_2\) & = B_E(z,r_1) \setminus \overline{B_E(z,r_2)}, \quad \text{for} \; r_1 > r_2 > 0.\\
	A\(\ex(z); r_2,r_1\) & = \ex\(A_E \(z; r_2,r_1\)\), \quad \text{for} \; r_1 > r_2 > 0.
\end{align*}

Here, we state the change of variable formula for the exponential map at $0 \in T_0\(\bn\)$. The proof can be found in the appendix of this article.\\

\begin{lemma}\label{Change of variable}
	Let $\Omega$ be an open subset of $\bn$ and $u : \bn \to \R$ be a measurable function. Then
	\begin{align}
		\int_{\Omega} u \dv & = \int_{\ex^{-1}\(\Omega\)} u \(\ex (z)\) \cdot \Upsilon(z) \mathrm{~d}z \label{COV} 
	\end{align}
	assuming both the integrals have finite values and
	\begin{align*}
		\Upsilon(z) & = 2 \[\f{\sinh(2|z|)}{|z|}\]^{N-1}.
	\end{align*}
\end{lemma}

\textbf{Hyperbolic translation:} For any $b \in \bn$, the hyperbolic translation $\tau_b : \bn \rightarrow \bn$ is defined by
\begin{align}
	\tau_b(x) &= \f{\(1-|b|^2\)x + \(|x|^2 + 2x \cdot b + 1\)b} {|b|^2|x|^2 + 2 x \cdot b + 1}. \label{hyperbolic translation}
\end{align}
It acts as a translation along the line $\(- \f{b}{|b|}, \f{b}{|b|}\)$. Also, we have $\tau_b(0) = b, \, \forall\, b \in \bn$. For further details on this topic, we advise referring to the book \cite{MR4221225}.
 
In this manuscript, we will denote the  gradient vector field and the Laplace-Beltrami operator on $\bn$ by $\nabla_{\bn} $ and $\De_{\bn}$ respectively. Also, the hyperbolic volume element by $\mathrm{~d}V_{\bn}$. Define the Sobolev space
\begin{align*}
	H^1(\bn) & := \{u \in L^2 (\bn) \;:\; \int_{\bn} \abs{\nabla_{\bn} u}^2 \mathrm{~d} V_{\bn} < \infty\},
\end{align*}
with the norm  
\begin{align*}
	\norm{u}_\la := \[\int_{\bn} \[\abs{\nabla_{\bn} u}^2 - \la u^2\] \mathrm{~d} V_{\bn}\]^{\f{1}{2}},
\end{align*}
 for any $\la < \f{(N-1)^2}{4}$. This is an equivalent norm on $H^1\(\bn\)$, and let us denote the corresponding inner product by  $\< \cdot, \cdot \>_\la$. Next, we present a lemma regarding hyperbolic translations. The proof can be found in \cite{MR3495548}.\\
\begin{lemma} \label{change of variable for translation}
	Let $u, v \in H^1(\bn)$. Then:
	\begin{align*}
		(i) \;& \int_{\bn} \< \nabla_{\bn} \(u \circ \tau_b\)\,,\, \nabla_{\bn}\(v \circ \tau_b\) \>_{\bn} \mathrm{~d} V_{\bn} = \int_{\bn}  \<\(\nabla_{\bn} u\) \circ \tau_b \,,\, \(\nabla_{\bn} v\) \circ \tau_b\>_{\bn} \mathrm{~d} V_{\bn}.\\
		(ii)\; & \int_{\bn} \(u \circ \tau_b\)(x) v(x) \mathrm{~d} V_{\bn} = \int_{\bn} u(x) \(v \circ \tau_{-b}\)(x) \mathrm{~d} V_{\bn}.
	\end{align*}
	Moreover, for any open subset $\Omega$ of $\bn$ and a measurable function $u$, we obtain
	\begin{align*}
		\int_{\Omega} \abs{u \circ \tau_b}^p \mathrm{~d} V_{\bn} & = \int_{\tau_b\(\Omega\)} \abs{u}^p \mathrm{~d} V_{\bn}, \quad 1 \leq p < \infty,
	\end{align*}
	provided the integrals are finite.
\end{lemma}


\subsection{An admissible subspace and variational framework}


In this section we shall introduce an equivariant subspace of sign-changing functions and look for solutions of \ef{(1)} in this equivariant space. To achieve this, let us first consider the set of block diagonal matrices $G$ as in (\ref{isom.grp}). Then we can easily see $G \subset O(N)$ is not a subgroup but has the following important properties:
	\begin{enumerate}
		\item Each $g \in G$ preserves the hyperbolic distance, defined in (\ref{hyperbolic metric}), i.e.,
		\begin{align*}
			d_{\bn} \(x,y\) = d_{\bn} \(gx,gy\), \, \forall x,y \in \bn, g \in G.
		\end{align*} 
		\item For each $g = \begin{bmatrix}
			A & 0\\
			0 & -1
		\end{bmatrix} \in G$, there exists $g' = \begin{bmatrix}
		A^{-1} & 0\\
		0 & -1
		\end{bmatrix} \in G$ such that $g g' = g'g = \mathrm{Id}_N$, where $ \mathrm{Id}_N$ denotes the $N \times N$ real identity matrix.
	\end{enumerate}

\vspace*{4mm}
Now let us define a map $T_g : H^{1}(\bn) \to H^{1}(\bn)$ such that
\begin{align}
	\(T_g u\) (x) & := - u \(g x\), \quad \forall g \in G. \label{the map T_g}
\end{align}
Then, for $g, g' \in G$, $g g' = g'g = \mathrm{Id}_N$ implies $T_g T_{g'} = T_{g'} T_g =  \mathrm{Id}$, where $\mathrm{Id}$ is the identity map on $H^{1}(\bn)$.

\begin{lemma}
	For each $g \in G$, $T_g$ preserves the inner product in $H^1\(\bn\)$.
\end{lemma}
\begin{proof}
	The proof follows easily from the fact that the volume form $\dv$ and the gradient field $\nabla_{\bn}$ are invariant under the actions of the elements of $G$.
\end{proof}
Now let us define a subspace of $H^1\(\bn\)$ to be
\begin{align}
	\widetilde{H^1\(\bn\)} & := \{u \in H^1\(\bn\) \; : \; T_g u = u, \;\; \forall g \in G\} \notag \\
	& = \{u \in H^1\(\bn\) \; : \; u(g x) = - u(x), \;\; \forall g \in G, x \in \bn \}. \label{definition of subspace}
\end{align}
Also, from the definition of exponential map (\ref{exponential map}), it is easy to observe that
\begin{align}
	\ex(gz) = g \, \ex(z), \quad \forall g \in G, z \in T_0 \(\bn\) \cong \R^N. \label{interaction between g and exponential}
\end{align}
Therefore, we can redefine $\widetilde{H^1\(\bn\)}$ to be
\begin{align}
	\widetilde{H^1\(\bn\)} = \{u \in H^1(\bn) \; : \; u\(\ex(gz)\) = - u\(\ex(z)\), \quad \forall g \in G, z \in \R^N\}. \label{equivalent definition}
\end{align}
\begin{remark}
	Let $x = \(x_1, \cdots, x_N\) \in \bn$. Then, for any $u \in \widetilde{H^1\(\bn\)}$,
	\begin{align*}
		u(x_1, \cdots, x_{N-1}, - x_N) & = - u (x_1, \cdots, x_{N-1}, x_N)
	\end{align*}
	Since $S \(0,r\)  = S_E \(0, \tanh \(\f{r}{2}\)\),$ any nontrivial function $u \in \widetilde{H^1\(\bn\)}$ is non-radial and also sign changing. Furthermore, it is easy to observe that $\widetilde{H^1\(\bn\)}$ is a non-trivial subspace of $H^1\(\bn\)$.
\end{remark}

\vspace*{4mm}
Now let $\{u_n\} \subset \widetilde{H^1(\bn)}$ such that $u_n \to u$ in $H^1\(\bn\)$. Then $u_n \to u$, and $\nabla_{\bn}u_n \to \nabla_{\bn}u$ in $L^2\(\bn\)$ as $n \to \infty$. Now $u_n\in \widetilde{H^1\(\bn\)}$ implies $u_n(gx) = -u_n(x), \; \forall g \in G$. Then, using the dominated convergence theorem, we get $u(gx)=-u(x), \; \forall g \in G$. And we have the following:
%
%
%
%
\begin{lemma}
	 $\widetilde{H^1\(\bn\)}$ is a closed subspace of $H^1\(\bn\)$.
\end{lemma}\;

The following result establishes the relationship between the hyperbolic translations in the $N$th direction and the action of $G$.

\begin{lemma}\label{relationship between translation and action}
	Let $b = \(0, \cdots, 0, b_N\) \in \bn$, then
	\begin{align*}
		u \circ \tau_b (gx) = - u \circ \tau_{-b} (x), \quad \forall x \in \bn, u \in \widetilde{H^1\(\bn\)}.
	\end{align*} 
\end{lemma}
\begin{proof}
	For every $g \in G$, $u \in \widetilde{H^1\(\bn\)}$, and $x = \(x_1, \cdots, x_N\) \in \bn$, we have
	\begin{align*}
		u \circ \tau_b (gx) & = u \(\f{\(1-|b|^2\)(g x) + \(|g x|^2 + 2(g x) \cdot b + 1\)b} {|b|^2|g x|^2 + 2 (g x) \cdot b + 1}\)\\
		& = u \(\f{\(1-|b|^2\)(g x) + \(|x|^2 - 2 x_N \cdot b_N + 1\) \(g\(0, \cdots, 0, - b_N\)\)} {|b|^2| x|^2 - 2 x_N \cdot b_N + 1}\)\\
		& = u \(g \(\f{\(1-|-b|^2\)x + \(|x|^2 + 2x \cdot (-b) + 1\)(-b)} {|-b|^2|x|^2 + 2 x \cdot (-b) + 1}\) \)\\
		& = - u \(\f{\(1-|-b|^2\)x + \(|x|^2 + 2x \cdot (-b) + 1\)(-b)} {|-b|^2|x|^2 + 2 x \cdot (-b) + 1}\)\\
		& = - u \circ \tau_{-b} (x).
	\end{align*}
\end{proof}
Let us now consider the following example: Let $\phi \in \widetilde{H^1\(\bn\)}$ such that $\mathrm{Supp}\, \phi \subset \ex \(R^{N-1} \times (-1,1)\).$ Define a sequence of points in $T_0\(\bn\)$ as $z_n := \(0, \cdots, 0, 2n \)$, and take $x_n = \ex(z_n)$. Then from the \cref{change of variable for translation}, it is easy to observe that
\begin{align*}
	\phi_n:= \phi \circ \tau_{x_n} + \phi \circ \tau_{-x_n} \in H^1\(\bn\), \quad \forall n \in \N.
\end{align*}
Also, because of disjoint supports of $ \phi \circ \tau_{x_n}$ and $ \phi \circ \tau_{-x_n}$ we have $\forall n$, $\norm{\phi_n}_\la^2 = 2 \norm{\phi}_\la^2.$ Now from the \cref{relationship between translation and action}, for every $g \in G$ we have
\begin{align*}
	\phi_n (gx) & = \phi \circ \tau_{x_n} (gx) + \phi \circ \tau_{ - x_n} (gx)
 = - \phi \circ \tau_{ - x_n} (x) - \phi \circ \tau_{x_n} (x)
	 = - \phi_n (x).
\end{align*}
Since all the integrands have disjoint supports, we can easily show that for every $m,n \in N$ with $m \neq n$, we have $\norm{\phi_m - \phi_n}_{L^{p+1}} = 4^{\f{1}{p+1}} \norm{\phi}_{L^{p+1}}$. Which implies $\widetilde{H^1\(\bn\)}$ is not compactly embedded in $L^{p+1} \(\bn\)$ for $1<p<2^* - 1$.\\

\textbf{The Variational Framework:} As (\ref{(1)}) is superlinear, the corresponding energy functional is unbounded from below. So, we intend to use the constrained minimization method to find a solution, which is indeed helpful because of the difference in homogeneity of linear and non-linear terms. Our goal is to find a non-radial sign-changing solution, so we pose the variational problem in the subspace $\widetilde{H^1\(\bn\)}$. Precisely, we want to find a minimizer of the energy functional $\Psi : \widetilde{H^1\(\bn\)} \to \R$ defined as
\begin{align}
	\Psi(u) = \int_{\bn} \[\f{1}{2} \abs{\nabla_{\bn} u}^2 - \f{\la}{2} u^2 \] \mathrm{~d} V_{\bn} = \f{1}{2} \norm{u}_\la^2\label{en.fn}
\end{align}
restricted on the submanifold
\begin{align}
	M = \{u \in \widetilde{H^1\(\bn\)} \; :\; \norm{u}_{L^{p+1}\(\bn\)} = 1\}.\label{neh.set}
\end{align}

However, to prove the existence of a constrained minimizer, we need some form of compactness for the minimizing sequence. As our problem is set up on $\bn$, which is an entire unbounded space, so  we do not get compactness from the Rellich-Kondrakov theorem. Even though we have radial symmetry on the $(x_1, \cdots, x_{N-1})$-hyperplane, the lack of compactness happens through the hyperbolic translations in the $x_N$ direction, as discussed above. To show the existence of a minimizer, we use the concentration compactness principle \cite{MR0778970} by P. L. Lions, and we establish the lack of compactness does not happen by showing neither the minimizing sequence slips to infinity nor it breaks into parts that are going infinitely far away from each other.

\section{Existence theorems }

In the first part of this section, we show the existence of a constrained minimizer using Lions' concentration compactness principle. As the principle described on $\rn$, we use the change of variable formula (\ref{COV}) for the exponential map at $0 \in T_0\(\bn\)$ to translate the variational setup to the Euclidean space $T_0\(\bn\) \cong \rn$. Here, we want to prove the following theorem:

\begin{theorem} \label{Checking concentration compactness}
	Let $\{u_n\} \subset M$ be a minimizing sequence for $\Psi$. There exists a subsequence $\{u_{n_k}\}$ and a corresponding sequence of points $\{x^k\}$ in $\bn$ such that for every $\eps > 0$ there exists $R(\eps) > 0$ so that
	\begin{align*}
		\int_{B\(x^k, R(\eps)\)} \abs{u_{n_k}}^{p+1} \dv \geq 1 - \eps.
	\end{align*}  
\end{theorem}
Before proving this theorem, we recall the following lemma, which can be derived similarly as in Lemma I.1 of \cite{MR778974}.

\begin{lemma}\label{Lieb}
	Let $1 <p < 2^* - 1$ and $\{v_n\}$ be a bounded sequence in $H^1\(\bn\)$ such that for some $R >0$,
	\begin{align*}
		\lim_{n \to \infty} \sup_{x \in \bn} \int_{B(x,R)} \abs{v_n}^{p+1} \dv = 0.
	\end{align*}
	Then 
	\begin{align*}
		\lim_{n \to \infty} \int_{\bn} \abs{v_n}^{p+1} \dv = 0.
	\end{align*}
\end{lemma}

\begin{proof}[Proof of \cref{Checking concentration compactness}] 	We have $\{u_n\} \subset M$ to be a minimizing sequence. Then from (\ref{COV}), we have
	\begin{align*}
		\int_{\R^N} |u_n\(\ex(z)\)|^{p+1} \;\Upsilon(z) \mathrm{~d}z & = 1
	\end{align*}
	Let us consider a sequence of functions $\rho_n : T_0(\bn) \cong \R^N \to \R$ such that
	\begin{align}
		\rho_n(z) := |u_n\(\ex(z)\)|^{p+1} \;\Upsilon(z) \label{CCS}
	\end{align}
	It is easy to observe that
	\begin{align*}
		\rho_n \in L^1(\R^N), \quad \rho_n \geq 0, \quad \text{and} \quad \int_{\R^N} \rho_n(z)  \mathrm{~d}z = 1
	\end{align*}
	Now we show that $\{\rho_n\}$ does not satisfy vanishing and dichotomy conditions in the concentration compactness principle.\\
	
	\textbf{Case I: Vanishing does not hold.}  If possible, let there exist a subsequence $\{\rho_{n_k}\}$ of $\{\rho_n\}$ such that
	\begin{align}
		\lim\limits_{k \to \infty} \; \sup_{\tilde{z} \in \rn} \int_{B_E(\tilde{z},R)} \rho_{n_k}(z) \mathrm{~d}z = 0, \;\; \forall R < \infty.
	\end{align}
	From (\ref{COV}), we have
	\begin{align*}
		\lim\limits_{k \to \infty} \; \sup_{\tilde{z} \in \rn} \int_{B(\ex(\tilde{z}), R)} \abs{u_{n_k}}^{p+1} \dv = 0.
	\end{align*}
	Now from the \cref{Lieb}, we have
	\begin{align*}
		\lim_{k \to \infty} \int_{\bn} \abs{u_{n_k}}^{p+1} \dv = 0
	\end{align*}
	This is a contradiction.
	
	\textbf{Case II: Dichotomy does not hold.} For every $\eta > 0$, we set
	\begin{align*}
		\Psi_\eta = \inf \{\Psi(u) \;\vert \; u \in \widetilde{H^1\(\bn\)}\,,\, \int_{\bn} |u|^{p+1} \dv = \eta\}
	\end{align*}
	Poincar\'e-Sobolev inequality implies $\Psi_\eta$ is finite for every positive real number $\eta$. Also, using the homogeneity of the norm, it is easy to observe that for every $\eta > 0$,
	\begin{align}
		\Psi_\eta = \eta^{\f{2}{p+1}} \Psi_1.\label{homoginity}
	\end{align}
	If possible, let the dichotomy hold, i.e., there exists a subsequence of $\{\rho_n\}$ denoted as $\{\rho_{n_k}\}$ and $ \al \in (0 , 1)$ such that for all $\eps > 0$, there exist a positive real number $R \equiv R(\eps)$, a sequence of points $\{z^k\} \subset \rn$, $\{R_k \equiv R_k \(\eps\)\} \subset \R$, and $k_0 \in \N$ such that for every $k \geq k_0$ we have
	\begin{align}
		R_k > R+3, \label{Dichotomy 2}
	\end{align}
	and
	\begin{align}
		\rho_{k}^1 := \rho_{n_k} \rchi_{B_E\(z^k,R\)}; \quad \quad \rho_k^2 := \rho_{n_k} \rchi_{\rn \setminus B_E\(z^k,R_k\) } \label{Dichotomy support}
	\end{align}
	satisfy the following:
	\begin{align}
		\norm{\rho_{n_k} - \(\rho_k^1 + \rho_k^2\)}_{L^1} < \eps, \quad \abs{\int_{\rn} \rho_k^1 \mathrm{~d}z - \al} < \eps, \quad  \abs{\int_{\rn} \rho_k^2 \mathrm{~d}z - (\la - \al)} < \eps \label{Dichotomy 1}
	\end{align}
	Furthermore, we have $R_k \to + \infty$ as $k \to \infty$. Note, from (\ref{Dichotomy 2}) and (\ref{Dichotomy support}), it can be observed that for every $k \geq k_0$,
	\begin{align}
		\mathrm{Supp} \; \rho_k^1 \cap \mathrm{Supp} \; \rho_k^2 = \phi. \label{Support disjoint}
	\end{align}
	We denote $x^k := \ex \(z^k\).$ Now, we define two smooth functions on $\bn$ to be
	\begin{align*}
		\chi_k^1  = \begin{cases}
			1 & \text{if} \quad x \in B\(x^k,R\),\\
			0 & \text{if} \quad x \in \bn \setminus B\(x^k,R+1\)
		\end{cases}
	\quad	\text{and} \quad \chi_k^2  = \begin{cases}
			1 & \text{if} \quad x \in \bn \setminus B\(x^k,R+3\),\\
			0 & \text{if} \quad x \in B\(x^k,R+2\)
		\end{cases}
	\end{align*}
	such that $0 \leq \chi_k^1, \chi_k^2 \leq 1$ and $\abs{\nabla_{\bn} \chi_k^1}, \abs{\nabla_{\bn} \chi_k^2} \leq 1$. 
	Now, let us define
	\begin{align*}
		u_k^1 := u_{n_k} \chi_k^1, \quad & u_k^2 := u_{n_k} \chi_k^2; \\
		\text{and,} \quad \beta_k := \int_{\bn} \abs{u_k^1}^{p+1} \dv, \quad & \gamma_k :=  \int_{\bn} \abs{u_k^2}^{p+1} \dv.
	\end{align*}
	From (\ref{Dichotomy 1}) and (\ref{Support disjoint}), for $k \geq k_0$ we have
	\begin{align*}
		 \int_{A_E \(z^k; R, R_k\)} \rho_{n_k} \mathrm{~d}z < \eps,
		\quad \text{and} \quad  \abs{\int_{B_E\(z^k, R\)} \rho_{n_k} \mathrm{~d}z - \al} < \eps.
	\end{align*}
	Using (\ref{COV}), for every $k \geq k_0$,
	\begin{align}
		\int_{A \(x^k; R, R+3\)} \abs{u_{n_k}}^{p+1} \dv < \eps,
		 \quad\text{and} \quad & \abs{\int_{B(x^k, R)} \abs{u_{n_k}}^{p+1} \dv - \al } < \eps. \label{Estimate of u_n_k on annulus}
	\end{align}
	Combining the above two inequalities, for every $k \geq k_0$ we have
	\begin{align}
		& \abs{\int_{\bn} \abs{u_k^1}^{p+1} \dv - \al} < 2 \eps \implies 
		 \abs{\beta_k - \al} < 2 \eps. \label{beta_k convergence}
	\end{align}
	Similarly, we can have
	\begin{align}
		\abs{\gamma_k - \(1- \al\)} < 2 \eps. \label{gamma_k convergence}
	\end{align}
	Let us denote
	\begin{align*}
		A_k^1 := B\(x^k, R+1\), \quad & B_k^1 := A \(x^k; R, R+1\);\\
		\text{and} \quad A_k^2 := \bn \setminus \overline{ B\(x^k, R+2\)}, \quad & B_k^2 := A \(x^k; R+2, R+3\).
	\end{align*}
	Then, $A_K^1 \cap A_k^2 = \phi, \quad \text{and} \quad B_k^1 \cap B_k^2 = \phi.$ Now for $i =1,2$ we have
	\begin{align*}
		\int_{\bn} \abs{\nabla_{\bn} u_k^i}^2 \dv & \leq \int_{A_k^i}  \abs{\nabla_{\bn} u_{n_k}}^2 \dv + \int_{B_k^i} \abs{u_{n_k}}^2 \dv\\ + & 2 \int_{B_k^i}  \<\chi_k^i \nabla_{\bn} u_{n_k}\; , \;u_{n_k} \nabla_{\bn} \chi_k^i \> \dv.
	\end{align*}
	Therefore we get
	\begin{align*}
		\int_{\bn} \abs{\nabla_{\bn} u_k^1}^2 \dv + \int_{\bn} \abs{\nabla_{\bn} u_k^2}^2 \dv \leq & \int_{A_k^1 \cup A_k^2} \abs{\nabla_{\bn} u_{n_k}}^2 \dv + \int_{B_k^1 \cup B_k^2} \abs{u_{n_k}}^2 \dv + \\
		& \; \qquad 2 \sum_{i=1}^{2} \int_{B_k^i}  \<\chi_k^i \nabla_{\bn} u_{n_k}\; , \;u_{n_k} \nabla_{\bn} \chi_k^i \> \dv.
	\end{align*}
	
	Now for every $k \geq k_0$, from (\ref{Estimate of u_n_k on annulus}) we have
	\begin{align*}
		\int_{B_k^1 \cup B_k^2} \abs{u_{n_k}}^2 \dv & \leq \int_{A \(x^k; R, R+3\)} \abs{u_{n_k}}^2 \dv\\
		& \leq C(R) \int_{A \(x^k; R, R+3\)} \abs{u_{n_k}}^{p+1} \dv \leq C(R) \eps.
	\end{align*}
	Since $\{\abs{\nabla_{\bn} u_{n_k}}\}$ is uniformly bounded in $L^2\(\bn\)$, for every $k \geq k_0$ we have
	\begin{align*}
		2 \sum_{i=1}^{2} \int_{B_k^i}  \<\chi_k^i \nabla_{\bn} u_{n_k}\; , \;u_{n_k} \nabla_{\bn} \chi_k^i \> \dv & \leq C \int_{B_k^1 \cup B_k^2} \abs{u_{n_k}}^2 \dv  \leq C(R) \eps.
	\end{align*}
	Hence, for every $k \geq k_0$ we get
	\begin{align}
		\int_{\bn} \abs{\nabla_{\bn} u_k^1}^2 \dv + \int_{\bn} \abs{\nabla_{\bn} u_k^2}^2 \dv & \leq \int_{\bn} \abs{\nabla_{\bn} u_{n_k}}^2 \dv + C(R) \eps \label{estimate for gradient norm}
	\end{align}
	
	For $\la \leq 0$, we will get
	\begin{align}
		- \la \int_{\bn} \abs{u_{n_k}}^2 \dv \geq - \la \int_{\bn} \abs{u_k^1}^2 \dv -\la \int_{\bn} \abs{u_k^2}^2 \dv, \quad \forall \, k \geq k_0. \label{lambda less than 0}
	\end{align}	
	Now for $\la > 0$, we can argue that for every $k \geq k_0$
	\begin{align}
		- \la \int_{\bn} \abs{u_k^1}^2 \dv - \la \int_{\bn} \abs{u_k^2}^2 \dv \leq - \la \int_{\bn} \abs{u_{n_k}}^2 \dv + \la C(R)\eps, \quad \forall \, k \geq k_0. \label{lambda greater than 0}
	\end{align}
	Hence, combining (\ref{lambda less than 0}) and (\ref{lambda greater than 0}), we get for every $k \geq k_0$,
	\begin{align}
		- \la \int_{\bn} \abs{u_k^1}^2 \dv - \la \int_{\bn} \abs{u_k^2}^2 \dv \leq - \la \int_{\bn} \abs{u_{n_k}}^2 \dv + \la_1 C(R)\eps. \label{estimate for square norm}
	\end{align}
	Therefore, from (\ref{estimate for gradient norm}) and (\ref{estimate for square norm}) we have
	\begin{align*}
		\Psi \(u_k^1\) + \Psi \(u_k^2\) \leq \Psi \(u_{n_k}\) + C(R)\eps, \quad \forall \, k \geq k_0.
	\end{align*}
	Here $C(R)$ is a generic constant that only depends on $R$. Since $\eps >0$ is arbitrary, taking $k \to \infty$ and using (\ref{beta_k convergence}), (\ref{gamma_k convergence}) we have
	\begin{align}
		\Psi_\al + \Psi_{1-\al} \leq \Psi_1  \label{Psi alpha}
	\end{align}
	Suppose there exists $u \in \widetilde{H^1\(\bn\)}$ such that 
	\begin{align*}
		\int_{\bn} \abs{u}^{p+1} \dv = \al, \; \text{i.e.,} \int_{\bn} \abs{\f{u}{\al^{\f{1}{p+1}}}}^{p+1} \dv & = 1
	\end{align*}
	Now from the definition we have
	\begin{align*}
		\Psi_1 & \leq \Psi \( \f{u}{\al^{\f{1}{p+1}}} \) = \f{1}{\al^{\f{2}{p+1}}} \Psi(u)
	\end{align*}
	This implies
	\begin{align*}
		\al^{\f{2}{p+1}} \Psi_1 \leq \inf \{ \Phi(u) \,:\, u \in \widetilde{H^1\(\bn\)}, \int_{\bn} |u|^{p+1} \dv = \al\} = \Psi_\al.
	\end{align*}
	Similarly, we can show that
	\begin{align*}
		\(1- \al\)^{\f{2}{p+1}} \Psi_1 & \leq \Psi_{1 - \al}
	\end{align*}
	Now from (\ref{Psi alpha}), we observe that
	\begin{align*}
		\al^{\f{2}{p+1}} + \(1- \al\)^{\f{2}{p+1}} & \leq 1
	\end{align*}
	This is impossible for $p > 1$ and $\al \in (0,1)$. Therefore, dichotomy does not hold. \\
	
	Now the only possibility is concentration, i.e., there exists $\{z^k\} \subset T_0\(\bn\) \cong \R^N$ such that
	\begin{align}
		\forall \eps > 0, \; \exists R(\eps) > 0, \;& \int_{B_E\(z^k, R(\eps)\)} \rho_{n_k} \mathrm{~d}z \geq 1 - \eps, \label{Concentration}
	\end{align}
	Let us denote $x^k := \ex \(z^k\).$ Therefore we have
	\begin{align*}
		\forall \eps > 0, \; \exists R(\eps) > 0, \;& \int_{B\(x^k, R(\eps)\)} \abs{u_{n_k}}^{p+1} \dv \geq 1-\eps.
	\end{align*}
\end{proof}

Now, our next goal is to show that $\{x^k\}$ is a bounded sequence in $\bn$. To show this, we use the actions of $T_g$. Then we can use the Rellich-Kondrakov theorem to show the existence of a minimizer. As in the above theorem, let us denote
\begin{align*}
	z^k = \(z^k_1, z^k_2, \cdots , z^k_{ N }\)& =  \ex^{-1} \(z^k\) \quad \forall k \in \N.
\end{align*}
Also, denote $\overline{z^k} := \(z^k_1, z^k_2, \cdots , z^k_{ N - 1}\) \in \R^{N-1}.$
\begin{lemma} \label{Bound in Nth direction}
	The sequence $\{z_N^k\}$ is bounded in $\R$.
\end{lemma}
\begin{proof}
	Let us denote $b_k  := \ex \(0, \cdots, 0, z_N^k\) = \(0, \cdots, 0, x_N^k\) \in \bn,$ and define
	\begin{align*}
		\tilde{u}_k := u_{n_k} \circ \tau_{b_k}
	\end{align*}
	where $\tau_{b_k}$ is the hyperbolic translation. From (\ref{Concentration}), we can write
	\begin{align*}
		& \int_{\R^{N-1} \times \(z_N^k - R(\eps), z_N^k + R(\eps)\)} \rho_{n_k} \mathrm{~d}z \geq 1-\eps\\
		\Rightarrow & \int_{\ex\(\R^{N-1} \times \(z_N^k - R(\eps), z_N^k + R(\eps)\)\)} \abs{u_{n_k}}^{p+1} \dv \geq 1 - \eps.
	\end{align*}
	Now from the \cref{change of variable for translation}, we have
	\begin{align}
		\int_{\ex \( \mathbb{R}^{N-1} \times \(- R\(\eps\), R\(\eps\)\)\)} \abs{\tilde{u}_k}^{p+1} \dv \geq 1 - \eps. \label{31}
	\end{align}
	From (\ref{interaction between g and exponential}), we have
	\begin{align*}
		& g \, \ex \(\(\R^{N-1} \times \(- R\(\eps\), R\(\eps\)\)\)\) = \ex \(\R^{N-1} \times \(- R\(\eps\), R\(\eps\)\)\), \quad \forall g \in G.
	\end{align*}
	Now (\ref{31}) becomes, for every $g \in G$
	\begin{align}
		\int_{\ex \( \mathbb{R}^{N-1} \times \(- R\(\eps\), R\(\eps\)\)\)} \abs{\tilde{u}_k (gx)}^{p+1} \dv \geq 1 - \eps. \label{32} 
	\end{align}
	From the \cref{relationship between translation and action} we have
	\begin{align*}
		\tilde{u}_k (gx) & =  u_{n_k} \circ \tau_{b_k} \(gx\)= - u_{n_k} \circ \tau_{-b_k} (x)= - \tilde{u}_k \circ \tau_{- b_k} \circ \tau_{- b_k} (x).
	\end{align*}
	Now from (\ref{32}) we get
	\begin{align}
\int_{\tau_{-b_k} \circ \, \tau_{-b_k} \(\ex \( \mathbb{R}^{N-1} \times \(- R\(\eps\), R\(\eps\)\)\)\)} \abs{\tilde{u}_k(x)}^{p+1} \dv \geq 1 - \eps. \label{33}
	\end{align}
	If possible, let $\{z_N^k\}$ have a subsequence, still denoted by $\{z_N^k\}$, such that $z_N^k \to \infty$ as $k \to \infty$. This implies corresponding $b_k \to \infty$ in $\bn$. Then for sufficiently large $k$ we have
	\begin{align*}
		\ex \( \mathbb{R}^{N-1} \times \(- R\(\eps\), R\(\eps\)\)\) \cap \tau_{-b_k} \circ \, \tau_{-b_k} \(\ex \( \mathbb{R}^{N-1} \times \(- R\(\eps\), R\(\eps\)\)\)\) = \phi
	\end{align*}
	Now from (\ref{32}) and (\ref{33}) we observe
	\begin{align*}
		& \int_{\bn} \abs{\tilde{u}_k(x)}^{p+1} \dv  \geq 2(1 - \eps)
		\Rightarrow  \int_{\bn} \abs{u_k(x)}^{p+1} \dv  \geq 2(1 - \eps)
	\end{align*}
	This is a contradiction to the fact that $u_k \in M$.
\end{proof}

\begin{lemma}\label{Bound in 1 to N-1 the directions}
	$\{ \overline{z^k} \}$ is bounded in $\R^{N-1}$.
\end{lemma}
\begin{proof}
	From (\ref{Concentration}), we can easily see
	\begin{align*}
		\int_{B_E\(\overline{z^k}, R(\eps)\) \times \R} \rho_{n_k} \mathrm{~d}z \geq 1 - \eps,
	\end{align*}
	where $B_E\(\overline{z^k}, R(\eps)\)$ is a ball of radius $R(\eps)$ and centered at $\overline{z^k}$ inside $\R^{N-1}$. Also, let us recall
	\begin{align*}
		\rho_{n_k} (z) = \abs{u_{n_k} \(\ex(z)\)}^{p+1} \Upsilon(z), \;\; \text{and} \; \Upsilon(gz) = \Upsilon(z), \; \forall g \in G.
	\end{align*}
	If possible, let $\overline{z^k} \to \infty$ in $\R^{N-1}$. Then for large enough $k$, there exists $h \in O(N-1)$ such that
	\begin{align*}
		h \(B_E\(\overline{z^k}, R(\eps)\)\) \cap B_E\(\overline{z^k}, R(\eps)\) = \phi.
	\end{align*}
	Now let us define 
	\begin{align*}
		\hat{g} := \begin{bmatrix}
			h & 0 \\
			0 & -1
		\end{bmatrix}
	\end{align*}
	It is easy to observe that
	\begin{align}
		\hat{g} \(B_E\(\overline{z^k}, R(\eps)\) \times \R\) \cap B_E\(\overline{z^k}, R(\eps)\) \times \R= \phi \label{34}
	\end{align}
	Now from the change of variable formula, we have
	\begin{align}
		\int_{B_E\(\overline{z^k}, R(\eps)\) \times \R} \rho_{n_k} \(\hat{g} z\) \mathrm{~d}z & = \int_{\hat{g} \(B_E\(\overline{z^k}, R(\eps)\) \times \R\)} \rho_{n_k} (z) \mathrm{~d}z \label{35}
	\end{align}
	Now
	\begin{align*}
		\int_{B_E\(\overline{z^k}, R(\eps)\) \times \R} \rho_{n_k} \(\hat{g} z\) \mathrm{~d}z  = \int_{B_E\(\overline{z^k}, R(\eps)\) \times \R} \abs{ u_{n_k} \(\ex(z)\)}^{p+1} \Upsilon(z)  \mathrm{~d}z
		& \geq 1 - \eps.
	\end{align*}
	Therefore, from (\ref{34}) and (\ref{35}) we have
	\begin{align*}
		\int_{\bn} \abs{u_{n_k}}^{p+1} \dv & = \int_{\R^N} \rho_{n_k} \mathrm{~d}z\\
		& \geq  \int_{B_E\(\overline{z^k}, R(\eps)\) \times \R} \rho_{n_k}(z) \mathrm{~d}z + \int_{\hat{g} \(B_E\(\overline{z^k}, R(\eps)\) \times \R\)} \rho_{n_k} (z) \mathrm{~d}z \\
		& \geq 2(1 - \eps).
	\end{align*}
	This is a contradiction, since $\{u_{n_k}\} \subset M$.
\end{proof}
Now we prove the existence of a constrained minimizer.\\

\begin{theorem}
	There exists a constrained minimizer $\hat{u}$ of $\Psi$ over $M$.
\end{theorem}
\begin{proof}
	Let $\{u_n\}$ be a minimizing sequence for the constrained minimization. Then from the \cref{Checking concentration compactness}, we have that  there exists $z^k \in \R^N$ such that
	\begin{align*}
		\forall \eps > 0, \; \exists R(\eps) > 0, \;& \int_{B_E\(z^k, R(\eps)\)} \rho_{n_k} \mathrm{~d}z \geq 1 - \eps
	\end{align*}
	From \cref{Bound in Nth direction}, we have $\{z_N^k\}$ is bounded in $\R$, and from \cref{Bound in 1 to N-1 the directions}, we obtain $\{\overline{z^k}\}$ is bounded in $\R^{N-1}$, i.e., there exists a compact set $K \subset \R^N$ containing the origin, such that $\{z^k\} \subset K$. Therefore, we have
	\begin{align*}
		\int_{B_E\(K, R(\eps)\)} \rho_{n_k} \mathrm{~d}z \geq 1 - \eps, 
	\end{align*}
	where $B_E\(K, R(\eps)\) := \{z \in \R^N \, : \, \mathrm{dist} \(K, z\) < R(\eps)\}$. Now from (\ref{COV}) we have
	\begin{align}
		\int_{\ex\(B_E\(K, R(\eps)\)\)} \abs{u_{n_k}}^{p+1} \dv \geq 1 - \eps, \quad \forall \eps > 0. \label{36}
	\end{align}
	Observe that $\{u_{n_k}\}$ is also a minimizing sequence. It has a subsequence, still denoted as $\{u_{n_k}\}$, such that $u_{n_k} \rightharpoonup \hat{u} \; \text{in} \; \widetilde{H^1\(\bn\)}.$ From the Rellich compactness theorem, we get for each $\eps > 0$,
	\begin{align*}
		u_{n_k} \to \hat{u} \quad \text{in} \; L^{p+1} \(\ex\(B_E\(K, R(\eps)\)\)\).
	\end{align*}
	Now (\ref{36}) implies
	\begin{align*}
		\int_{\bn} \abs{\hat{u}}^{p+1} \dv \geq 1.
	\end{align*}
	And from weak lower semi-continuity of the $L^{p+1}$ norm, we have
	\begin{align*}
		\int_{\bn} \abs{\hat{u}}^{p+1} \dv \leq \liminf_{k \to \infty} \int_{\bn} \abs{u_{n_k}}^{p+1} \dv = 1
	\end{align*}
	Therefore
	\begin{align*}
		\int_{\bn} \abs{\hat{u}}^{p+1} \dv = 1.
	\end{align*}
	This implies $\hat{u} \in M$. Lastly, from the weak lower semi-continuity of the Sobolev norm, we have
	\begin{align*}
		\Psi(\hat{u}) \leq \liminf_{k \to \infty} \Psi\(u_{n_k}\).
	\end{align*}
	Hence $\hat{u}$ is a constrained minimizer for $\Psi$.
\end{proof}
Finally, we can prove the existence theorem:

\begin{proof}[Proof of \cref{Existence thm hyperbolic}]{\;}
	From the previous theorem, we have $\hat{u}$ is a constrained minimizer of $\Psi : \widetilde{H^1\(\bn\)} \to \R$ defined as
	\begin{align*}
		\Psi(u) = \int_{\bn} \[\f{1}{2} \abs{\nabla_{\bn} u}^2 - \f{\la}{2} u^2 \] \mathrm{~d} V_{\bn} = \f{1}{2} \norm{u}_\la^2
	\end{align*}
	restricted on the submanifold
	\begin{align*}
		M = \{u \in \widetilde{H^1\(\bn\)} \; :\; \norm{u}_{L^{p+1}\(\bn\)} = 1\}.
	\end{align*}
	Now by Lagrange's multiplier theorem, we have
	\begin{align*}
		\< \hat{u}, v \>_\la - \mu \int_{\bn} \abs{\hat{u}}^{p-1} \hat{u} v \dv, \quad \text{for some} \; \mu \in \R^+ \; \text{and} \; \forall v \in \widetilde{H^1\(\bn\)}.
	\end{align*}
	Now consider a function $I : \widetilde{H^1\(\bn\)} \to \R$ to be
	\begin{align*}
		I(u) : = \f{1}{2} \norm{u}_\la^2 - \f{\mu}{p+1} \int_{\bn} \abs{u}^{p+1} \dv.
	\end{align*}
	Therefore, $\hat{u}$ is a critical point of $I$. Now we take the extension of $I$ in $H^1(\bn)$ to be $J : H^1(\bn) \to \R$ such that
	\begin{align*}
		J(u) : = \f{1}{2} \norm{u}_\la^2 - \f{\mu}{p+1} \int_{\bn} \abs{u}^{p+1} \dv.
	\end{align*}
	
	
	By the change of variable formula, we have that $J$ is invariant under the actions of $T_g, \; g \in G$, i.e.,
	\begin{align*}
		J\(T_g u \) & = J(u), \quad \forall g \in G, u \in H^1\(\bn\).
	\end{align*}
	Now for each $g \in G$, we have
	\begin{align*}
		DJ(T_g u) \[v\] & = \lim\limits_{t \to 0} \f{J(T_g u + t v) - J(T_g u)}{t}
		 = \lim\limits_{t \to 0} \f{J\(u + t T_{g'}v\) - J(u)}{t}
		 = DJ(u) \[T_{g'}v\] 
	\end{align*}
	Let $\nabla J(u)$ be the gradient of $J$ at $u$. Now, we can see $\nabla J(T_g u) = T_g \(\nabla J(u)\), \forall g \in G.$ Therefore, $\nabla J$ is equivariant. Since $\widetilde{H^1\(\bn\)}$ is a closed subspace of the Hilbert space $H^1\(\bn\)$, we can write
	\begin{align*}
		H^1\(\bn\) & = \widetilde{H^1\(\bn\)} \oplus \widetilde{H^1\(\bn\)}^\perp.
	\end{align*}
	Since $\hat{u} \in \widetilde{H^1\(\bn\)}, \; T_g \hat{u} = \hat{u}, \; \forall g \in G $. And, similarly, we have $\nabla J(\hat{u}) \in \widetilde{H^1\(\bn\)}.$ Since $\hat{u}$ is a critical point of $I$, we have $\nabla J(\hat{u}) \in \widetilde{H^1\(\bn\)}^{\perp}.$  Therefore, we can argue that $\nabla J(\hat{u})$ must be zero, i.e., $\hat{u}$ is a critical point of $J$. So, $\hat{u}$ is a critical point of $J$, i.e., $\hat{u}$ solves
	\begin{align*}
		-\Delta_{\mathbb{B}^N} u  \, - \,  \lambda  u &= \mu |u|^{p-1}u \quad \text{in} \; H^1\(\bn\).
	\end{align*}
	Then $w : = \mu^{\f{1}{p-1}} \hat{u}$ solves (\ref{(1)}). Since $\hat{u} \in \widetilde{H^1\(\bn\)}$, it is non-radial and sign-changing. Therefore, $w$ is also non-radial and sign-changing.
	
\end{proof}

Now, we use the \cref{Existence thm hyperbolic} to prove the \cref{Existence thm HSM}.

\begin{proof}[Proof of \cref{Existence thm HSM}]
	Let $v$ be a sign-changing solution to (\ref{(1)}) established in \cref{Existence thm hyperbolic}, then from the relationship 
	\begin{align*}
		 u(y,r) = r^{\f{2-n}{2}} v \circ M^{-1} (y,r), \;\; (y,r) \in \rn_+ = \R^{h+1}_+,
	\end{align*}
	it is obvious that the corresponding solution $u$ to (\ref{HSM equation}) is also sign-changing in nature.\\
	
	Let $\overline{g} \in O(h) = O(N-1)$ be arbitrary and $v \in \widetilde{H^1\(\bn\)}$ be a solution to (\ref{(1)}). Consider
	\begin{align*}
		g = \begin{bmatrix}
			\overline{g} & 0\\
			0 & -1
		\end{bmatrix}, \; 
		& g' = \begin{bmatrix}
			\mathrm{Id}_{N-1} & 0\\
			0 & -1
		\end{bmatrix} \in G.
	\end{align*} 
	Now 
	\begin{align*}
		u(\overline{g}y,r) & = r^{\f{2-n}{2}} v \circ M^{-1} (\overline{g}y,r)\\
		& = r^{\f{2-n}{2}} v \( \f{2\overline{g}y}{\abs{y}^2 + \(1 + r\)^2}, \f{1-\abs{(y,r)}^2}{\abs{y}^2 + \(1 + r\)^2} \)\\
		& = r^{\f{2-n}{2}} v \(g \, g' \( \f{2y}{\abs{y}^2 + \(1 + r\)^2}, \f{1-\abs{(y,r)}^2}{\abs{y}^2 + \(1 + r\)^2} \)\)\\
		& = r^{\f{2-n}{2}} v \circ M^{-1} (y,r) = u(y,r)
	\end{align*}
	Hence, $u$ is a bi-radial sign-changing solution to (\ref{HSM equation}).
\end{proof}

\section{Multiplicity theorems}

Here we recall the examples of groups $\Ga$ and corresponding continuous onto homomorphisms from \cite{MannaManna+2025} to define the subspace $H^1\(\bn\)^\phi$ as in (\ref{phi equivariant subsapce}). First, we take $\tau \(x_1,x_2,x_3,x_4, \cdots, x_N\) := \(x_3,x_4,x_1,x_2,x_5,\cdots,x_N\).$. We divide the examples into two cases: For the case $N=5$, we take the group $\Ga$ to be
\begin{align*}
	& \Ga := \mathrm{Span} \{O(2) \otimes O(2) \otimes O(N-4), \tau\},
\end{align*}
and, for $N = 4\; \& \; N \geq 6$, we consider \begin{align*}
	& \Ga := \mathrm{Span} \{O(2) \otimes O(2) \otimes \{\mathrm{Id}\}, \tau\}.
\end{align*}
And, we take $\phi : \Ga \to \mathbb{Z}_2$ as
\begin{align*}
	\phi(g) & = 1,\;\forall g \in \Ga, \quad \text{ and } \quad
	\phi(\tau)  = -1.
\end{align*}
We note, in this case, the conditions \eqref{(A_1)} hold for the point $(1/2,0,\dots,0)$ and \eqref{(A_2)} holds as $\mathrm{Fix}_{\bn}(\Ga)=\{0\}$. 

In the following lemma we prove that the subspaces $\widetilde{H^1\(\bn\)}$ and $H^1\(\bn\)^\phi$ have only trivial intersection.
\begin{lemma}\label{zero intersection of two subspaces}
	For $N \geq 4$, $\widetilde{H^1\(\bn\)} \cap H^1(\bn)^\phi= \{0\}$.
\end{lemma}
\begin{proof}
	For the cases $ N = 4$ and $N \geq 6$, the lemma follows easily as $\tilde{G} = G \cap O(2) \otimes O(2) \otimes O(N-4) \neq \emptyset.$ Then for any $ u \in \widetilde{H^1\(\bn\)} \cap H^1(\bn)^\phi$, we have for any $g_1\in \tilde{G}$
	\begin{align*}
		&u (g_1 x) = u(x), \quad \forall x \in \bn &&\text{ as } u \in u \in H^1(\bn)^\phi,\\
		\text{Also}, \quad &u (g_1 x) = - u(x), \quad \forall x \in \bn, &&\text{ as } u \in \widetilde{H^1\(\bn\)}.
	\end{align*}
	Hence $u\equiv 0$. For the case $N=5$, take $g\in O(2) \otimes O(2) \otimes \{\mathrm{Id}\}, \, h\in G$ as
	\begin{align*}
		g  = \begin{bmatrix}
			0 & 1 & 0 & 0 & 0\\
			1 & 0 & 0 & 0 & 0\\
			0 & 0 & 0 &1 & 0\\
			0 & 0 & 1 & 0 & 0\\
			0 & 0 & 0 & 0 & 1
		\end{bmatrix} \quad  \text{ and } \quad
		h  = \begin{bmatrix}
			0 & 0 & 0 & 1 & 0\\
			0 & 0 & 1 & 0 & 0\\
			0 & 1 & 0 &0 & 0\\
			1 & 0 & 0 & 0 & 0\\
			0 & 0 & 0 & 0 & -1
		\end{bmatrix} 
	\end{align*}  
	If $ \exists \, u \(\neq 0\) \in \widetilde{H^1\(\bn\)} \cap H^1(\bn)^\phi$, then we can easily see
	\begin{align*} u(z)=-u(\tau (z))=-u(g\circ \tau (z))=-u(h\circ g\circ\tau (z))=u(z',-z_5)=-u(z).\end{align*}
	Hence $u=0$.
\end{proof}
Now, we prove the multiplicity theorem.

\begin{proof}[Proof of \cref{Multiplicity thm hyperbolic}]{\;}
	From \cite{MannaManna+2025}, we have that for $N \geq 4$ there exists a non-trivial solution of (\ref{(1)}) in $H^1(\bn)^\phi$. And from the \cref{Existence thm hyperbolic} there exists a non-trivial solution of (\ref{(1)}) in $\widetilde{H^1\(\bn\)}$. Now, since both the subspaces $H^1(\bn)^\phi$ and $\widetilde{H^1\(\bn\)}$ contain only non-radial sign-changing functions except zero, from \cref{zero intersection of two subspaces} we can argue that (\ref{(1)}) has two non-radial sign-changing solutions in $H^1\(\bn\)$.
\end{proof}


\section{Appendix}

Here we prove the change of variable formula (\ref{COV}) for the exponential map, $\ex : T_0 \bn \to \bn$ given by
\begin{align*}
	\mathrm{exp}_0(z) & = \f{\sinh \(2|z|\)}{1+ \cosh \(2|z|\)} \f{z}{|z|},  \quad \forall z \in T_0 \(\bn\)\cong \rn.
\end{align*}
Let us first mention a well-known determinant identity.\\

\begin{lemma}{(Sylvester's determinant identity)}
	Let $A$ and $B$ be two matrices of sizes $m \times n$ and $ n\times m$ respectively, then
	\begin{align*}
		\det \(I_m + AB\) & = \det \(I_n + BA\).
	\end{align*}
\end{lemma}

\begin{proof}[Proof of \cref{Change of variable}]
	Let us denote 
	\begin{align*}
		\al(z) = \f{\sinh \(2|z|\)}{1+ \cosh \(2|z|\)}.
	\end{align*}
	Then, 
	\begin{align*}
		\f{\partial \al}{\partial z_j} =  \f{2z_j}{|z| \(1 + \cosh (2|z|)\)}
	\end{align*}
	Now
	\begin{align*}
		J_{ij}(z) & = \f{\partial}{\partial z_j} \( \al (z) \f{z_i}{|z|} \)  = (A-C)z_iz_j + B \delta_{ij},
	\end{align*}
	where
	\begin{align*}
		A  = \f{2}{|z|^2 \(1 + \cosh (2|z|)\)}, \quad
		B  = \f{\sinh(2|z|)}{|z|\(1 + \cosh (2|z|)\)}, \quad \text{ and }\quad
		C  = \f{\sinh(2|z|)}{|z|^3\(1 + \cosh (2|z|)\)}.
	\end{align*}
	Therefore, the Jacobian is
	\begin{align*}
		\det \(J(z)\) & = \det \(B I + \(A-C\) z z^T\), \quad B > 0\\
		& = B^N \cdot (1 + z^T \cdot B^{-1}(A-C)z)\\
		& = B^N \(1 +B^{-1}(A-C) |z|^2 \)\\
		& = B^N + (A - C) B^{N-1} |z|^2\\
		& = \[\f{\sinh(2|z|)}{|z|\(1 + \cosh (2|z|)\)}\]^N + \\
		& \;\;\;\;\; \[\f{2}{|z|^2 \(1 + \cosh (2|z|)\)} - \f{\sinh(2|z|)}{|z|^3\(1 + \cosh (2|z|)\)}\] \[\f{\sinh(2|z|)}{|z|\(1 + \cosh (2|z|)\)}\]^{N-1} |z|^2\\
		& = \[\f{\sinh(2|z|)}{|z|\(1 + \cosh (2|z|)\)}\]^{N-1} \[\f{\sinh(2|z|)}{|z|\(1 + \cosh (2|z|)\)} + \f{2|z| - \sinh(2|z|)}{|z|\(1 + \cosh (2|z|)\)}\]\\
		& = \[\f{\sinh(2|z|)}{|z|\(1 + \cosh (2|z|)\)}\]^{N-1} \cdot \f{2}{1 + \cosh (2|z|)}.
	\end{align*}
	We have that $\Omega$ is an open subset of $\bn$ and $u : \bn \to \R$ be a measurable function. Then
	\begin{align*}
		\int_{\Omega} u \dv & = \int_{\ex^{-1}\(\Omega\)} u \(\ex (z)\) \( \f{2}{1 - |\ex(z)|^2}\)^N \[\f{\sinh(2|z|)}{|z|\(1 + \cosh (2|z|)\)}\]^{N-1} \cdot \f{2}{1 + \cosh (2|z|)} \mathrm{~d}z\\
		& = \int_{\ex^{-1}\(\Omega\)} u \(\ex (z)\) \(1 + \cosh(2|z|)\)^N \[\f{\sinh(2|z|)}{|z|\(1 + \cosh (2|z|)\)}\]^{N-1} \cdot \f{2}{1 + \cosh (2|z|)} \mathrm{~d}z\\
		& = \int_{\ex^{-1}\(\Omega\)} u \(\ex (z)\) \cdot 2 \[\f{\sinh(2|z|)}{|z|}\]^{N-1} \mathrm{~d}z,
	\end{align*}
	assuming that the integrals have finite values. \\
	
	Hence the change of variable formula for the exponential map at $0$ is
	\begin{align*}
		\int_{\Omega} u \dv & = \int_{\ex^{-1}\(\Omega\)} u \(\ex (z)\) \cdot \Upsilon(z) \mathrm{~d}z 
	\end{align*}
\end{proof}

Finally, we prove the \cref{Multiplicity thm HSM}, which follows from the findings in \cite{MannaManna+2025}. In this article, we maintained the constraint on dimension as $2<k<n$, where $\R^n = \R^h \times \R^k$. For the proof, we use the least dimension; specifically, we take $k=3$. For the solutions of (\ref{Existence thm HSM}), the $\R^k$-coordinates become radial, and we keep the radial dimension as it is. We take involution $\tau$ of the $\R^h$-coordinates, where $h \geq 4$. Post-involution, the points reside within the same sphere in $R^n$, although the solution will have different signs at those two points. Thus, we prove the theorem for $n \geq 7$.
\begin{proof}[Proof of \cref{Multiplicity thm HSM}]{\;}
	(a) To prove this part, we use the $N=5$ case as in \cite{MannaManna+2025}. Suppose $v \in H^1\(\bn\)^\phi$ is a solution to (\ref{(1)}). Then we have
	\begin{align*}
		u (y,z) = u(y,r) = r^{\f{2-n}{2}} v \( \f{2y}{\abs{y}^2 + \(1 + r\)^2}, \f{1-\abs{(y,r)}^2}{\abs{y}^2 + \(1 + r\)^2} \)
	\end{align*}
	solves (\ref{HSM equation}), where $y = (y_1, \cdots, y_4) \in \R^4, z = (z_1, \cdots, z_3)\in \R^3$ such that $\abs{z} = r$. Let us define $\tau(y,z) = (y_3, y_4, y_1, y_2, z_1, \cdots, z_3) = (y',z)$. Then $(y,z)$ and $\tau(y,z)$ belong to the same sphere in $\R^7$, but
	\begin{align*}
		u\(\tau(y,z)\) = u (y',z) = u(y', r) & = r^{\f{2-n}{2}} v \( \f{2y'}{\abs{y}^2 + \(1 + r\)^2}, \f{1-\abs{(y,r)}^2}{\abs{y}^2 + \(1 + r\)^2} \)\\
		& = r^{\f{2-n}{2}} v \(\tau \( \f{2y}{\abs{y}^2 + \(1 + r\)^2}, \f{1-\abs{(y,r)}^2}{\abs{y}^2 + \(1 + r\)^2} \)\)\\
		& = - r^{\f{2-n}{2}} v \( \f{2y}{\abs{y}^2 + \(1 + r\)^2}, \f{1-\abs{(y,r)}^2}{\abs{y}^2 + \(1 + r\)^2} \)\\
		& = - u(y,z)
	\end{align*}
	Therefore, $u$ is a non-radial sign-changing solution to (\ref{Existence thm HSM}).\\

	(b) To prove this part, we use, $N \geq 6$ case as in \cite{MannaManna+2025}. In this case $y = (y_1, \cdots, y_{n-3}) \in \R^{n-3}$ and $z = (z_1, \cdots, z_3)\in \R^3$. And the involution is defined as
	\begin{align*}
		\tau(y,z) = (y_3, y_4, y_1, y_2, y_5, \cdots, y_{n-3}, z_1, \cdots, z_3)
	\end{align*}

	Suppose $\{v_k\} \subset H^1\(\bn\)^\phi$ is a sequence of solutions to (\ref{(1)}). Then we can make similar constructions as above, to get a sequence of non-radial sign-changing solutions $\{u_k\}$ to (\ref{HSM equation}).
\end{proof}

\bibliographystyle{plain}

\end{document}